\documentclass[12pt]{amsart}

\usepackage{a4}
\usepackage{amsmath, amssymb}
\usepackage{mathrsfs}
\usepackage{exscale}
\usepackage[varg]{txfonts}
\usepackage{enumerate}

\renewcommand{\iff}{\Leftrightarrow}

\newcommand{\ox}{\otimes}

\newcommand{\<}{\langle}
\renewcommand{\>}{\rangle}

\DeclareMathOperator*{\lperp}{\Huge \textrm {$\perp$}}

\newcommand{\vf}{\varphi}

\newcommand{\gL}{\Lambda}
\newcommand{\Kd}{{K^{\times}}}
\newcommand{\Kdt}{{\Kd^{2}}}

\newcommand{\ohm}{\omega}

\newcommand{\HP}{\mathbb{H}}
\newcommand{\nlo}{\frac{n-1}{2}}
\newcommand{\nslo}{\frac{n^2-1}{2}}

\newcommand{\nslfo}{\frac{n^2-4}{2}}
\newcommand{\nslno}{\frac{n^2-n}{2}}
\newcommand{\kpow}{\gL^k}
\newcommand{\hyp}{\mathrm {Hyp}}
\newcommand{\lb}{\left(}
\newcommand{\rb}{\right)}

\newtheorem{lemma}{Lemma}[section]

\newtheorem{propo}[lemma]{Proposition}
\newtheorem{coro}[lemma]{Corollary}

\theoremstyle{definition}

\newtheorem{defi}[lemma]{Definition}

\newtheorem{remark}[lemma]{Remark}
\newtheorem{example}[lemma]{Example}

\title{Trace forms of Symbol Algebras}
\author{Ronan Flatley}
\date{November 27, 2008}             
\address{School of Mathematical Sciences, University College Dublin, Belfield, Dublin~4, Ireland}
\email{ronan.flatley@ucd.ie}

\begin{document}
\maketitle

\begin{abstract} Let $S$ be a symbol algebra. The trace form of $S$ is computed
 and it is shown how this form can be used to determine whether $S$ is a division algebra or not. In addition, the exterior powers of the trace form of $S$ are computed.
 \end{abstract}

\section{Introduction}

Let $n$ be an arbitrary positive integer and let $K$ be a field containing a primitive $n$-th root of unity $\omega$. Unless stated otherwise, we assume throughout this paper that ${\rm char}(K)$ is different from $2$ and does not divide $n$. Let $\Kd=K\backslash\{0\}$. Let $a$,$b \in$ $\Kd$ and 
 let $S$ be the algebra over $K$ generated by elements $x$ 
 and $y$ where $$x^n = a, \quad y^n=b \quad\text{ and } \quad yx=\omega x y.$$ 
 We call this algebra a {\it symbol algebra} (see \cite[Chapter $1$, \S$2$]{KMRT}). Note that in \cite[\S11]{D}, Draxl calls such an algebra a {\it power norm residue algebra}, denoting it as $(a, b; n, K, \omega)$ and shows it to be a central simple algebra over $K$ of degree $n$. Quaternion algebras are the symbol algebras of degree $2$.

\par Let $A$ be a central simple algebra of degree $n$ over a field $K$ of characteristic different from $2$. We write $T_A \colon A \rightarrow K$ for the quadratic trace form $$T_A(z)={\text {Trd}}_A(z^2) {\quad\text{for }} z \in A,$$ where Trd$_A$ is the reduced trace of $A$. The main purpose of this paper is to compute the trace form of a symbol algebra $S$ and to show how the form determines if the algebra is division.

\par Notation and terminology is borrowed from Lam's book \cite{TYL} and Scharlau's book \cite{S}. A diagonalised quadratic form over $K$ with coefficients $a_1, \dots, a_m \in \Kd$ is denoted by $\<a_1, \dots, a_m\>$. The hyperbolic plane $\<1, -1\>$ is denoted by $\HP$.
If $\vf$ and $\psi$ are forms over $K$ then $\vf \simeq \psi$ means that these forms are isometric. 
The Witt index of $\vf$ is denoted by $i_W(\vf)$ and the anisotropic part of 
$\vf$ by $\vf_{an}$, so that we have $\vf \simeq \vf_{an}\perp \lb i \times \mathbb {H} \rb$ where $i=i_W(\vf)$. The tensor product of the $1$-dimensional form $\<n\>$ with $\vf$ is denoted $\<n\>\vf$ to be distinguished from $n$ copies of $\vf$ which is written as $n \times \vf$.

\section{Symbol algebras and their trace forms}\label{sec.s.a.}

Let $S$ be the symbol algebra $(a,b; K,n, \omega)$ with basis $\{x^iy^j\}$, $0 \le i, j \le n-1$.

\begin{propo}\label{P1}
We have
\begin{align}
\text{(i)\quad} T_S &\simeq \<n\> \perp \nslo\times\HP &&\text{ for } n \text{ odd.}\notag\\
\text{(ii)\quad}T_S &\simeq \<n\>\<1,a,b,(-1)^{n/2}ab\> \perp \nslfo \times \HP &&\text{ for } n \text{ even.}\notag
\end{align}
\end{propo}

\begin{proof} 
Let $\phi_{T_S}$ be the symmetric bilinear form associated with $T_S$. Consider $\{x^iy^j\}$, $0 \le i,j \le n-1$, the set of $n^2$ basis elements of $S$. Consider the left regular matrix representation of each such element under the isomorphism used in the definition of the trace map. Easy, but tedious arguments, and switching from the trace to the reduced trace show that
$$\phi_{T_S}(x^iy^j, x^{i}y^{j})=\text {Trd}_S((x^iy^j)^2) = 0$$
unless we have of the following cases:
\begin{align}
i&=j=0, &&\text{in which case the reduced trace is $n$} ; \notag\\
i&=0 \text{ and } j=\frac{n}{2}, &&\text{in which case the reduced trace is $nb$};\notag\\
i&=\frac{n}{2} \text{ and } j=0, &&\text{in which case the reduced trace is $na$};\notag\\
i&=\frac{n}{2} \text{ and } j=\frac{n}{2}, &&\text{in which case the reduced trace is $(-1)^\frac{n}{2}nab$}.\notag
\end{align}
 Clearly, the latter three cases only arise for $n$ even. \\
\par{\it(i)} Le $i=j=0$. Then $\phi_{T_S}(1,1)=n$, as mentioned above. Now let $1 \le i, j \le n-1$. There are $\nslo$ pairs $(x^iy^j, x^{n-i}y^{n-j})$ and we have
$$\phi_{T_S}(x^iy^j, x^{n-i}y^{n-j})=\text {Trd}_S(x^iy^jx^{n-i}y^{n-j}) = \text {Trd}_S(\omega^{j(n-i)}x^ny^n) 
= n\omega^{-ij}ab.$$
Each pair contributes a $2\times2$ block in the Gram matrix of $\phi_{T_S}$, as follows: \begin{table}[h]
     \begin{center}
	   \begin{tabular}{c|cc}
	  $\phi_{T_S}(-,-)$      	& $x^iy^j$ & $x^{n-i}y^{n-j}$\\ \hline
	$x^iy^j$ & $0$		   & $n\omega^{-ij}ab$\\
       $x^{n-i}y^{n-j}$ & $n\omega^{-ij}ab$  & $0$\\ 
       \end{tabular}
    \end{center}
\end{table}       \\
Each such block corresponds to a $2$-dimensional hyperbolic plane which is a direct summand of $T_S$. The Gram matrix will have exactly one non-zero entry in each row.
Hence $T_S \simeq \<n\> \perp \nslo\times\HP$.	  
\par{\it(ii)} The subset $\{ 1 , x^{n/2} , y^{n/2} , x^{n/2}y^{n/2} \}$ of basis elements of $S$ gives rise to the quadratic form $\<n\>\< 1 , a , b , (-1)^{n/2} ab \>$. By placing the other $n^2-4$ basis elements into ordered pairs of the form $( x^i y^j , x^{n-i} y^{n-j})$, we get that $\phi_{T_S}$ maps each pair to $n\ohm^{-ij}ab$ as seen in the proof of {\it(i)}. So we get $\nslfo$ hyperbolic planes as direct summands. Hence $T_S \simeq \<n\>\<1,a,b,(-1)^{n/2}ab\> \perp \nslfo \times \HP$.	
\end{proof}

By way of example, the matrix of $T_S$ when $n=3$ is as follows:

\begin{equation*}
     \left[
     \begin{smallmatrix}
          3	&	&	&	&	&	&	&	&	\\
            	& 0  	& 3a   &   	&   	&   	&   	&  	& 	\\
            	& 3a  	& 0   	&   	&   	&   	&  	&  	&	\\
            	&   	&    	& 0  	& 3b  	&   	&   	&  	&	\\
            	&   	&    	& 3b 	&  0	&   	&   	&  	&	\\
            	&   	&    	&   	&   	& 0  	&3\ohm^2ab&  	&	\\
            	&   	&  	&   	&   	&3\ohm^2ab& 0  	&  	&	\\
            	&   	&    	&   	&   	&   	&	& 0	& 3\ohm ab	&\\
            	&   	&    	&   	&   	&   	&	&3 \ohm ab 	&  0	
     \end{smallmatrix}
     \right].
\end{equation*}
The matrix is computed for the basis elements $\{1 , x , x^2 , y , y^2 , xy , x^2y^2 , x^2y , xy^2\}$ and each blank entry in the matrix is zero.

\section{Further results when $\deg_KS$ is odd}\label{sec.f.r.}
We now return to the case when $n$ is odd and show that we can improve upon the formula deduced for $T_S$. We require the following two propositions.

\begin{propo}\label{P2}
Let $n$ be odd. Then $\<n\> \simeq \<(-1)^\frac{n-1}{2}\>$. 
\end{propo}

\begin{proof}
We recall the following definitions from classical number theory: for $p$ an odd prime
$$\tau_p := \sum_{i=1}^{p-1} \lb \frac{i}{p} \rb \ohm^i$$
where $\lb \frac{i}{p}\rb $ is the Legendre symbol and
$$ p^* := \lb \frac{-1}{p}\rb p. $$ We have the theorems (see \cite{N}, for example)
$$p^*=\tau_p^2 \text{ and } \lb \frac{-1}{p}\rb =(-1)^\frac{p-1}{2}$$
and the facts
$$p_1 \equiv 1 \pmod{4} \iff p_1^*=p_1 \quad\text{ and }\quad p_2 \equiv 3 \pmod{4} \iff p_2^*=-p_2.$$
Suppose $n$ is prime. Then $\tau_n^2=n^*=\lb \frac{-1}{n} \rb n=(-1)^\frac{n-1}{2}n$. On the other hand suppose $n$ is compound. Write $n=\prod_{j=1}^{t} p_j$ where the $p_j$ are primes, not all necessarily different and $t > 1$. If $n \equiv 1 \pmod{4}$ then 
$$n=\prod_{j=1}^{t}p_j^* = (-1)^\frac{n-1}{2}\prod_{j=1}^{t}p_j^* = (-1)^\frac{n-1}{2}\lb \prod_{j=1}^t \tau_j \rb ^2.$$
Otherwise, if $n \equiv 3 \pmod{4}$ then
$$n=-\prod_{j=1}^{t}p_j^* = (-1)^\frac{n-1}{2}\prod_{j=1}^{t}p_j^* = (-1)^\frac{n-1}{2}\lb \prod_{j=1}^t \tau_j\rb ^2.$$
Hence $\<n\> \simeq \<(-1)^\frac{n-1}{2}\>$.
\end{proof}

\begin{propo}\label{P3}
Let $n$ be odd. Then $n \times \<1\> \simeq \<(-1)^\frac{n-1}{2}\> \perp \frac{n-1}{2} \times \HP$.
\end{propo}

\begin{proof}
Recall that the {\it level} of a field $F$, denoted $s(F)$,  is the least number of squares required to sum to $-1$ in $F$ or $\infty$ if no such number exists. Since $K$ contains a primitive $n$-th root of unity, $s(K) \in \{1, 2, 4\}$. If $n \equiv 3$ or $5\pmod{8}$, then $n$ has a prime divisor $p$ such that $p \equiv 3$ or $5 \pmod{8}$. In this case we have $s(K)=2$ by \cite{FGS}. Thus, $2 \times \<1\> \simeq 2 \times \<-1\>$ which implies
\begin{align}
\notag n \times \<1\> &\simeq 3 \times \<1\> \perp \nlo \times \HP &&\text{ if } n \equiv 3 \pmod{8} \\
\notag &\simeq \<-1\> \perp \nlo\times\HP\\
\notag n \times \<1\> &\simeq 5 \times \<1\> \perp \frac{n-7}{2}\times\HP &&\text{ if } n \equiv 5 \pmod{8} \\
\notag &\simeq \<1\> \perp \nlo\times\HP.
\end{align}
  In the remaining cases of interest we have $s(K) \le 4$, i.e.
$4 \times \<1\> \simeq 4 \times \<-1\>$. Therefore,
\begin{align}
\notag n \times \<1\> &\simeq \<1\> \perp \nlo \times \HP &&\text{ if } n \equiv 1 \pmod{8} \\
\notag n \times \<1\> &\simeq 7 \times \<1\> \perp \frac{n-7}{2}\times\HP &&\text{ if } n \equiv 7 \pmod{8} \\
\notag &\simeq \<-1\> \perp \nlo\times\HP.
\end{align}
\end{proof}

\begin{coro}
For $n$ odd, $T_S \simeq n^2 \times \<\lb -1 \rb ^\nlo\>$.
\end{coro}

\begin{proof}
Propositions \ref{P1}{\it(i)} and \ref{P2} show that for $n$ odd, 
\begin{align}
\notag T_S &\simeq \<n\> \perp \nslo \times \HP \\ 
\notag &\simeq \<(-1)^\nlo\> \perp \nslo \times \HP .
\end{align}
Then by Proposition \ref{P3} and the fact that 
$$n^2 \times \<(-1)^\nlo\> =  n \times \<(-1)^\nlo\> \ox n \times \<1\>,$$
we get $T_S \simeq n^2 \times \<\lb  -1 \rb ^\nlo\>$.
\end{proof}

\begin{remark}
For $n$ odd, we could also write $T_S \simeq n \times \<1\> \perp \nslno\times\HP$ which can be deduced from the calculation in the split case, see \cite{DWL}, together with Springer's Theorem on odd degree extensions (see \cite[p.194]{TYL}, for example).
\end{remark}

\section{Trace form criteria to determine if a symbol algebra is division}\label{sec.t.f.d.c.}

\begin{propo}\label{P4}
If $n \equiv 2 \pmod{4}$ then the quaternion algebra $(a, b)_K$ is contained in $S$.
\end{propo}

\begin{proof}
Consider the elements $x^{n/2}$,  $y^{n/2} \in S$. Then $\lb x^{n/2} \rb ^2=a$, $\lb y^{n/2}\rb ^2=b$ and $y^{n/2}x^{n/2}=\lb \ohm^{n/2}\rb ^{n/2}x^{n/2}y^{n/2}=-x^{n/2}y^{n/2}$.
\end{proof}

\begin{propo}\label{P5}
If $n \equiv 2 \pmod{4}$ and $T_S$ is hyperbolic then $-1$ is a square in $K$ and $S$ is not a division algebra.  
\end{propo}

\begin{proof}
Suppose $T_S$ is hyperbolic. By Proposition \ref{P4}, $Q:=(a,b)_K \subset S$. Since $n$ is even we have from Proposition \ref{P1} {\it(ii)} that
$$T_S \simeq \<n\>\<1, a, b, (-1)^{n/2}ab\> \perp \nslfo \times \HP \simeq \<n\>\<1, a, b, -ab\> \perp \nslfo \times \HP.$$
Thus, by our assumption, $\<1, a, b, -ab\> \simeq 2 \times \HP$ and by evaluating determinants we get $\<-1\> \simeq \<1\>$. Thus, the norm form of $Q$, $N_Q:=\<1, -a, -b, ab\>$ is hyperbolic. This implies that $Q$, and thus $S$, contains zero divisors. Therefore, $S$ is not division.
\end{proof}

\begin{propo}\label{P6}
Let $K$ be a field such that $-1 \in \Kdt$. Let $A$ be any central simple algebra over $K$. Let $n:=\deg_KA$ be a power of $2$. If $T_A$ is not hyperbolic, then $A$ is a division algebra.
\end{propo}

\begin{proof}
Suppose $A$ is not a division algebra. Then $A \cong M_r(D)$ for some integer $r>1$ and some division algebra $D$ over $K$. Now $T_A \simeq T_{M_r(D)} \simeq T_{M_r(K) \ox_K D} \simeq T_{M_r(K)} \ox T_D \simeq r \times \<1\> \ox T_D$ by \cite[Lemma 1.2]{DWL}. Since $n$ is a $2$-power, $r$ must be even and since $-1 \in \Kdt$, we have that $T_A$ is hyperbolic.
\end{proof}

\section{Exterior powers of the trace form of a symbol algebra}\label{sec.e.p.}

Bourbaki defined the concept of exterior power of a symmetric bilinear form in \cite[IX, \S1, (37)]{B}. McGarraghy derived basic properties of such forms in the Witt-Grothendieck ring of a field  in \cite{M}. We present some key definitions and results for exterior powers from McGarraghy's paper as well as some new results. In all cases, $K$ denotes a field of characteristic different from $2$.

\begin{defi}\label{D1}
Let $\vf: V \times V \rightarrow K$ be a bilinear form and let $k$ be a positive integer not greater than $m$. We define the {\it $k$-fold exterior power} of $\vf$, 
$$\kpow \vf: \kpow V \times \kpow V \rightarrow K$$
by 
$$\kpow \vf (x_1 \wedge \dots \wedge x_k, y_1 \wedge \dots \wedge y_k) = \det\big(\vf(x_i, y_j)\big)_{1 \le i, j \le k}.$$
We define $\gL^0\vf := \<1\>$, the identity form of dimension 1. For $k>m$, we define $\kpow \vf$ to be the zero form, since $\kpow V = 0$ for all $k>m$.
\end{defi}

Let $V$ be a vector space of dimension $m$ over $K$. If $k$ is a non-negative integer then the $k$-fold exterior power of $V$, $\kpow V$, has dimension $m \choose k$, where we take $m \choose k$ to be $0$ for all $k>m$. In particular, if $\{v_1, \dots , v_m\}$ is a basis for $V$, then a basis for $\kpow V$ is given by the set of $k$-fold wedge products $\{v_{i_1} \wedge \dots \wedge v_{i_k} : 1 \le i_1 < \dots < i_k  \le m \}$ and there are $m \choose k $ such expressions.

\begin{remark}
We have $\gL^1\vf \simeq \vf$.
It is easily seen that $\kpow \vf$ is a bilinear form and is symmetric if $\vf$ is symmetric.
Also, if $q$ is the quadratic form associated to $\vf$, we write $\kpow q$ for the quadratic form associated to $\kpow \vf$.
\end{remark}

\begin{propo}\cite[Proposition 4.1]{M}\label{S.P4.1}
Let $V$ be a vector space over $K$ with $\textup{dim}_KV=m$. Let $\vf$ be a symmetric bilinear form over $K$ with $\vf \simeq \<a_1, \dots , a_m\>$. Then $\kpow \vf$ is a symmetric bilinear form of dimension $m \choose k$ and 
$$\kpow \vf \simeq \lperp_{1 \le i_1 < \dots < i_k \le m}\<a_{i_1} \dots a_{i_k}\>.$$
In particular, $\kpow (m \times \<1\>) \simeq {m \choose k} \times \<1\>$.
\end{propo}

\begin{remark}
We also have that $\kpow (m \times \<-1\>) = {m \choose k} \times \<(-1)^k\>$.
\end{remark}

\begin{propo}\cite[Proposition 7.3]{M}\label{S.P7.3}
Let $\vf$ and $\psi$ be symmetric bilinear forms over $K$ and let $k \in \mathbb{N}$. Then 
$$\kpow(\vf \perp \psi)\simeq\lperp_{i+j=k}\gL^i\vf \ox \gL^j \psi$$
\end{propo}

\subsection{Exterior powers of hyperbolic forms}

We now compute exterior powers of a hyperbolic form $\phi \simeq h \times \HP$ where $h \in \mathbb{N}$.

\begin{propo}\label{P8}
Let $\phi \simeq h \times \HP$ where $ h \in \mathbb{N}$ and $ k$ odd with $1 \leq k \leq 2h-1$. Then
\[
\gL^k\phi \simeq \frac{1}{2} {2h \choose k} \times\HP.
\]
\end{propo}

\begin{proof} 
\begin{align*}
\kpow \phi 	&\simeq \kpow  \Big( h \times \< 1 \> \perp h \times \<-1\>\Big) \\
		&\simeq\lperp_{i+j=k}\gL^i  \Big( h \times \< 1 \> \Big) \ox \gL^j  \Big( h \times \< -1 \>\Big)\\ 
		&\simeq\lperp_{i+j=k} {h \choose i} \times \< 1 \> \ox {h \choose j} \times \< (-1)^j \>\\ 
		&\simeq \sum_{i \textup{ odd}}{h \choose i}{h \choose {k-i}} \times \<1\> \perp  \sum_{i \textup{ even}} {h \choose i}{h \choose {k-i}} \times \<-1\>.
\end{align*}
Since $$\sum_{i \textrm { odd}}{h \choose i}{h \choose k-i} = \sum_{i \textrm { even}}{h \choose i}{h \choose k-i}$$ for $k$ odd and since $\kpow\phi$ has dimension ${2h \choose k}$, the result follows.
\end{proof}

\begin{propo}\label{P9}
Let $\phi \simeq h \times \HP$ where $ h \in \mathbb{N}$, $k=2\ell$ and $0 \le \ell \le h$. Then 
\[
\kpow\phi = \gL^{2\ell}\phi \simeq  {h \choose \ell} \times \<(-1)^\ell\> \perp \frac{1}{2} \lb {2h \choose 2\ell}-{h \choose \ell}\rb \times\HP.
\]
\end{propo}

\begin{proof}
We use induction on $h$ and $\ell$. Let $P(h,\ell)$ be the statement in the proposition. $P(h, 1)$ is true for all $h$ since 
\begin{align*}
\gL^2\phi &\simeq 2{h \choose 2}\times \<1\> \perp {h \choose 1}^2 \times \<-1\> \\
&\simeq {h \choose 1}\times \<-1\> \perp \frac{1}{2}\lb {2h \choose 2}-{h \choose 1}\rb \times \HP.
\end{align*}
Consider $P(1,\ell)$. Here $\phi \simeq \HP \text{ and } \ell \in \{0,1\}$. Now 
\[
\gL^0\HP :=\<1\> \simeq {1 \choose 0}\times \<(-1)^0\> \perp \frac{1}{2}\lb {2 \choose 0}-{1 \choose 0}\rb \times \HP
\]
and 
\[
\gL^2\HP =\<-1\> \simeq {1 \choose 1}\times \<-1\> \perp \frac{1}{2}\lb {2 \choose 2}-{1 \choose 1}\rb \times \HP.
\]
So $P(1, \ell)$ is true.\\
{\it Inductive step}: Let $m, n$ be integers, $0 <m<h$. Assume $P(m, n-1)$ is true for $0 \le n-1 \le m$ or, equivalently,  $1 \le n \le m+1$. The case $n=m+1$ gives $\gL^{2(m+1)}:=\<0\>.$
Also, assume that $P(m-1, n)$ is true for $0 \le n \le m-1$. We prove $P(m, n)$ to be true for $0 \le n \le m$:
\begin{align*}
\gL^{2n}\phi &= \gL^{2n}\Big({m \times \HP}\Big) \\
&= \gL^{2n} \Big({(m-1) \times \HP \perp \HP}\Big)\\
&\simeq\lperp_{i+j=2n} \gL^i \big( {(m-1) \times \HP} \big) \ox \gL^j\HP\\
&\simeq\gL^{2n} \Big( (m-1) \times \HP \Big) \perp \gL^{2n-1} \Big( (m-1) \times \HP \Big) \ox \HP \perp \gL^{2n-2} \Big( (m-1) \times \HP \Big)\\
&\simeq {{m-1}\choose n} \times \<(-1)^n\> \perp \frac{1}{2}\lb {2(m-1)\choose 2n}-{{m-1}\choose n}\rb \times \HP \perp {{2(m-1)}  \choose {2n-1}}\times \HP \\
&\quad \perp {{m-1}\choose{n-1}}\times\<(-1)^n\> \perp \frac{1}{2}\lb {{2(m-1)}\choose{2(n-1)}}-{{m-1}\choose {n-1}}\rb \times \HP \\
&\simeq {m \choose n} \times \<(-1)^n\> \perp \frac{1}{2}\lb {2m \choose 2n}-{m \choose n}\rb \times \HP
\end{align*}
by $P(m-1,n)$, $P(m, n-1)$ and Proposition~\ref{P8}. 
\end{proof}

\begin{remark}
The above proposition has been proved for ordered fields in \cite[Proposition~11.8]{M}. The proof uses the signature of $\phi$ with respect to an ordering.
\end{remark}

\begin{remark}
In \cite{M} it was shown that when $K$ is an ordered field and $\phi$ a hyperbolic form then $\kpow \phi$ is hyperbolic if and only if $k$ is odd. Proposition \ref{P9}
 shows that this is not true for fields in general. For example, for a field $K$ containing $\sqrt{-1}$ and $\phi \simeq 4\times \HP$, we have $\gL^2\phi={4 \choose 1} \times \<-1\> \perp \frac{1}{2}\lb {8 \choose 2}-{4 \choose 1}\rb \times \HP = 14 \times \HP$.
 \end{remark}

\subsection{Some properties of binomial coefficients}

We use some properties of binomial coefficients in the subsequent section. Some of them are well-known, others not. All of them can be derived from first principles or by using identities to be found in \cite[Chapter 1]{R}, for example. We list the properties here:

\begin{align*}
&{{r}\choose {s}} + {{r}\choose{s-1}} = {{r+1}\choose{s}},\\\\
&{{r}\choose {s}} - {{r}\choose{s-1}} = \frac{r+1-2s}{r+1}{{r+1}\choose{s}},\\\\
&{{r}\choose {s}} + {{r}\choose{s-2}} = {{r+2}\choose{s}}-2{{r}\choose{s-1}},\\\\
&{{r}\choose {s}} - {{r}\choose{s-2}} = \frac{r+2-2s}{r+2}{{r+2}\choose{s}},\\\\
&{{r}\choose {s}} = \frac{r}{s}{{r-1}\choose{s-1}}.
\end{align*} 

\subsection{Computation of exterior powers of the trace form of a symbol algebra}

Let $K$ and $S$ be as in Section 1.
From Propositions \ref{P1}, \ref{P2} and \ref{P3}, we have that the trace form of $S$ is
\[
	T_S \simeq 
	\begin{cases}
		n \times \<1\> \perp \nslno \times \HP \simeq \<(-1)^\nlo\> \perp \nslo \times \HP, &\text{if $n$ is odd;}\\
		\<n\>\<1, a, b, (-1)^\frac{n}{2}ab\> \perp \nslfo \times \HP, &\text{if $n$ is even.}
	\end{cases}
\]
For the remainder of this section, we shall use ``$\hyp$" to denote an unspecified number of hyperbolic planes. Thus we may restate the computed trace form of $S$ as
\[
	T_S \simeq 
	\begin{cases}
		n \times \<1\> \perp \hyp \simeq \<(-1)^\nlo\> \perp \hyp, &\text{if $n$ is odd;}\\
		\<n\>\<1, a, b, (-1)^\frac{n}{2}ab\> \perp \hyp, &\text{if $n$ is even.}
	\end{cases}
\]

\begin{propo}\label{P10}
Let $n$ be odd and $k$ an integer such that $0 \le k \le n^2$. Then
\[
	\kpow T_S \simeq 
	\begin{cases}
		{\nslo \choose \frac{k-1}{2}} \times \<(-1)^\frac{k-1}{2}\> \perp \hyp, &\text{if $k$ is odd;}\\
		{\nslo \choose \frac{k}{2}}\times \<(-1)^\frac{k}{2}\> \perp \hyp, &\text{if $k$ is even.}
	\end{cases}
\]
\end{propo}

\begin{proof}
Let $k$ be odd. Then
\begin{align*}
	\kpow T_S &= \kpow \lb \nslo \times \HP \perp \<(-1)^\nlo \> \rb\\
	&\simeq \lperp_{i+j=k} \gL^i \lb \nslo \times \HP\rb \ox \gL^j\<(-1)^\nlo\>\\
	&\simeq \kpow \lb \nslo \times \HP \rb \perp \gL^{k-1} \lb \nslo \times \HP \rb \\
	&\simeq {\nslo \choose \frac{k-1}{2} } \times \<(-1)^{\frac{k-1}{2}}\> \perp \hyp.
\end{align*}
A similar computation for even $k$ yields the result.
\end{proof}

\begin{propo}\label{P11}
Let $n$ be even. We write $T_S \simeq q_S \perp m$ $\times$ $\HP$ where $q_S \simeq \<n\>\<1, a, b, (-1)^\frac{n}{2}ab\>$ and $m=\nslfo$. Then, for $0 \le k \le n^2$,
\[
	\kpow T_S \simeq 
	\begin{cases}
		{{m+1} \choose \frac{k-1}{2}} \times \<(-1)^\frac{n(k-1)}{4}\>q_S \perp \hyp, &\text{if $k$ is odd;}\\
		{\frac{n^2}{2} \choose \frac{k}{2}}\times \<1\> \perp \hyp, &\text{if $k$ is even and $n \equiv 0\pmod{4}$;}\\
		\lb 1-\frac{2k}{n^2}\rb{{\frac{n^2}{2}}\choose{\frac{k}{2}}} \times \<(-1)^\frac{k}{2}\> \perp \hyp, &\text{if $k$ is even, $k \le \frac{n^2}{2}$ and $n \equiv 2 \pmod{4}$;}\\	
		\lb \frac{2k}{n^2}-1\rb{{\frac{n^2}{2}}\choose{\frac{k}{2}}} \times \<(-1)^\frac{k+2}{2}\> \perp \hyp, &\text{if $k$ is even, $k>\frac{n^2}{2}$ and $n \equiv 2 \pmod{4}$.}
	\end{cases}
\]
\end{propo}

\begin{proof}
{\it Case (i)}. Let $k$ be odd. Then 
\begin{align*}
	\kpow T_S &= \kpow (m \times \HP \perp q_S) \\
	&\simeq \gL^{k-1} (m \times \HP) \ox q_S \perp \gL^{k-3}(m \times \HP) \ox \gL^3 q_S \perp 	\hyp \\
	&\simeq {m \choose \frac{k-1}{2}} \times \<(-1)^\frac{k-1}{2}\> \ox q_S \perp {m \choose \frac{k-3}{2}} \times \<(-1)^\frac{k-3}{2}\> \ox \<(-1)^\frac{n}{2}\>q_S \perp \hyp.
\end{align*}
When $n \equiv 0 \pmod{4}$ we have $-1 \in \Kdt$ and so 
\begin{align*}
	\kpow T_S &\simeq \lb{m \choose \frac{k-1}{2}}+{m \choose \frac{k-3}{2}}\rb \times q_S \perp \hyp = {{m+1} \choose \frac{k-1}{2}} \times q_S \perp \hyp.
\end{align*}
On the other hand, when $n \equiv 2 \pmod{4}$,
\begin{align*}
	\kpow T_S &\simeq {m+1 \choose \frac{k-1}{2}} \times \<(-1)^\frac{k-1}{2}\>q_S \perp \hyp.
\end{align*}
Hence, for $n$ even and $k$ odd, we have
 \begin{align*}
	\kpow T_S &\simeq {\frac{n^2-2}{2} \choose \frac{k-1}{2}} \times \<(-1)^\frac{n(k-1)}{4}\>q_S  \perp \hyp.
\end{align*}

{\it Case (ii)}. Let $k$ be even. Then 
\begin{align*}
	\kpow T_S &= \kpow (m \times \HP \perp q_S) \\
	&\simeq \kpow(m \times \HP) \perp \gL^{k-2} (m \times \HP) \ox \gL^2 q_S \perp \gL^{k-4}(m \times \HP) \ox \gL^4 q_S \perp 	\hyp \\
	&\simeq {m \choose \frac{k}{2}} \times \<(-1)^\frac{k}{2}\> \perp {m \choose 	\frac{k-4}{2}} \times \<(-1)^\frac{n}{2}\>\<(-1)^\frac{k}{2}\> \perp \hyp.
\end{align*}

When $n \equiv 0 \pmod{4}$ we have $-1 \in \Kdt$ and so 
\begin{align*}
	\kpow T_S &\simeq \lb{m \choose \frac{k}{2}}+{m \choose \frac{k-4}{2}}\rb \times \<1\> \perp \hyp \\
	&= {{m+2} \choose \frac{k}{2}} \times \<1\> \perp \hyp \\
	&= {\frac{n^2}{2} \choose \frac{k}{2}} \times \<1\> \perp \hyp.
\end{align*}

On the other hand, when $n \equiv 2 \pmod{4}$ we have that $\<(-1)^\frac{n}{2}\> \simeq \<-1\>$ and so
\begin{align*}
	\kpow T_S \simeq 
	\begin{cases}
		\lb {m \choose \frac{k}{2}} - {m \choose \frac{k-4}{2}}\rb \times \<(-1)^\frac{k}{2}\> \perp \hyp, &\text{if $k \le \frac{n^2}{2}$;}\\
		\lb {m \choose \frac{k-4}{2}} - {m \choose \frac{k}{2}}\rb \times \<(-1)^\frac{k+2}{2}\> \perp \hyp, &\text{if $k > \frac{n^2}{2}$}\\
	\end{cases} \\
	= 
	\begin{cases}
		\lb 1-\frac{2k}{n^2}\rb {\frac{n^2}{2} \choose \frac{k}{2}} \times \<(-1)^\frac{k}{2}\> \perp \hyp, &\text{if $k \le \frac{n^2}{2}$;}\\
		\lb \frac{2k}{n^2} - 1\rb {\frac{n^2}{2} \choose \frac{k}{2}} \times \<(-1)^\frac{k+2}{2}\> \perp \hyp, &\text{if $k > \frac{n^2}{2}$}.\\
	\end{cases}
\end{align*}
\end{proof}

\begin{example}
By Proposition \ref{P11} with $n=4$ and  $k$ odd, we get $\kpow T_S \simeq {7 \choose \frac{k-1}{2}} \times q_S \perp \hyp$. Since $7 \choose l$ is odd for $0 \le l \le 7$, it follows that $\kpow T_S \simeq q_S \perp \hyp$ for $k$ odd, $1 \le k \le 15$. For $n=4$ and $k$ even we get $\kpow T_S \simeq {8 \choose \frac{k}{2}}\times \<1\> \perp \hyp$. Since $8 \choose l$ is even for $1 \le l \le 7$ it follows that $\kpow T_S$ is hyperbolic. These conclusions confirm \cite[Corollaire 2]{RST}.
\end{example}

\begin{remark}
In general, for $n \equiv 0 \pmod{4}$ and $k$ even, it is not true that $\kpow T_S$ is hyperbolic. For example, with $n=12$, $\gL^{16}T_S\simeq \<1\> \perp \hyp$.
\end{remark}

\begin{remark}
From Proposition \ref{P11} it follows that $\kpow T_S$ is hyperbolic for $n$ even and $k \in \{n, \frac{n^2}{2}\}$. For, when $n \equiv 0 \pmod{4}$ we have 
\begin{align*}
\gL^\frac{n^2}{2} T_S &\simeq {\frac{n^2}{2} \choose \frac{n^2}{4}} \times \<1\> \perp \hyp \\ 
&= 2{\frac{n^2-2}{2}\choose \frac{n^2-4}{4}} \times \<1\> \perp \hyp
\end{align*}
and when $n \equiv 2 \pmod{4}$ the result follows directly from the formula for $\kpow T_S$.
\end{remark}

\begin{remark}
It follows from Proposition \ref{P11} that $\gL^{n^2} T_S$ is anisotropic when $n$ is even. 
\end{remark}

\begin{remark}
As a consequence of Proposition \ref{P11}, $\kpow T_S$ is hyperbolic when $n \equiv 0 \pmod{4}$, $p$ is an odd prime divisor of $n$ and $k \in \{2, 4, 8, 2p, 4p, 8p\}$.
\end{remark}

\section{Acknowledgements}

I would like to thank David Lewis and Thomas Unger for their guidance in the preparation of this paper and for their many suggested improvements. I am also grateful to Science Foundation Ireland who funded this research under the Research Frontiers Programme (project no. 07/RFP/MATF191).

\end{document}